\documentclass{amsart}
\usepackage{graphicx}
\usepackage{amsmath}
\usepackage{amsfonts}
\usepackage{amssymb}
\usepackage{enumitem}
\usepackage{ifpdf}
\providecommand{\U}[1]{\protect\rule{.1in}{.1in}}
\newtheorem{theorem}{Theorem}

\newtheorem{corollary}[theorem]{Corollary}

\newtheorem{definition}[theorem]{Definition}

\newtheorem{lemma}[theorem]{Lemma}

\newtheorem{proposition}[theorem]{Proposition}
\newtheorem{remark}[theorem]{Remark}

\newcommand{\Int}{\operatorname{Int}}

\makeatletter \@namedef{subjclassname@2020}{\textup{2020}
Mathematics Subject Classification} \makeatother
\begin{document}
\title{Oscillation Classes: An Interpolation Approach}
\author{Joaquim Mart\'{\i}n}
\address{Department of Mathematics\\
Universitat Aut\`onoma de Barcelona} \email{Joaquin.Martin@uab.cat}
\thanks{Partially supported by Grants PID2024-160507NB-I00 and PID2024-155917NB-I00
funded by MCIN/AEI/10.13039/501100011033.} 

\begin{abstract}
Let \(\varphi\) be an admissible concave function on \((0,1)\) and
let \(X\) be a rearrangement-invariant space. We study the classes
determined by the oscillation functional
\[
\mathcal N_{\varphi,X}(f) = \left\| \frac{f^{**}-f^*}{\varphi}
\right\|_X+\|f\|_1.
\]
We develop an interpolation method, based on the
Aronszajn--Gagliardo extremal construction, which allows the
normability problem and the determination of the optimal
rearrangement-invariant Banach exterior to be treated in a unified
way.

A recovery principle shows that the oscillation construction
reflects the inclusion order of the underlying
rearrangement-invariant spaces. This makes it possible to transfer
the corresponding Aronszajn--Gagliardo extremal structure to the
oscillation classes. In particular, the upper extremal generates the
least rearrangement-invariant Banach space containing the class,
while normability occurs precisely when the lower and upper
extremals collapse.

At the critical fundamental scale, normability is rigid and forces
the underlying space to be the corresponding Lorentz endpoint.
Applications to Lorentz, limiting Lorentz, and Orlicz scales
illustrate the scope of the method, including limiting normable
examples for which the classical Copson absorption mechanism fails.
\end{abstract}

\maketitle

%%%%%%%%%%%%%%%%%%%%%%%%%%%%%%%%%%%%%%%%%%%%%%%%%%%%%%%%%%%%%%%%%%%%%%%%%%%%%%%
\section{Introduction}
%%%%%%%%%%%%%%%%%%%%%%%%%%%%%%%%%%%%%%%%%%%%%%%%%%%%%%%%%%%%%%%%%%%%%%%%%%%%%%%

Let \(f\) be a measurable function on \((0,1)\), endowed with
Lebesgue measure.\footnote{The notation, terminology, and background
needed for the notions used in this introduction and throughout the
paper are recalled in Section~\ref{sec:background}.} If \(f^*\) and
\(f^{**}\) denote, respectively, the decreasing rearrangement of
\(f\) and its maximal average, then the quantity
\begin{equation}
\label{eq:intro-oscillation} f^{**}(t)-f^*(t), \qquad 0<t<1,
\end{equation}
which provides a natural measure of the oscillation of the
decreasing rearrangement of \(f\), has played a recurring role in
rearrangement theory, symmetrization, Sobolev embeddings, and
endpoint problems since the work of Bennett, DeVore and Sharpley
\cite{BennettDeVoreSharpley}. Its subsequent development and
applications can be found, among many others, in
\cite{Kolyada1989,BasteroMilmanRuiz2003,
CarroGogatishviliMartinPick2005,MartinMilmanPustylnik2007,
MartinMilman2010,Mastylo2011,MartinMilman2014,MartinOrtiz2023}, and
the references therein.

A basic difficulty in working directly with
\eqref{eq:intro-oscillation} is that this quantity is nonlinear.
Traditionally, this difficulty has been approached either by
identifying equivalent expressions that recover norms of classical
rearrangement-invariant spaces, or by using oscillation estimates
without requiring an underlying function-space structure. The
present paper is concerned precisely with the interaction between
these two viewpoints.

Let \(\varphi\) be an admissible concave function on \((0,1)\), and
let \(X\) be a rearrangement-invariant Banach function space (r.i.
space, for short). We consider the normalized oscillation
\[
\frac{f^{**}(t)-f^*(t)}{\varphi(t)}
\]
and the associated class
\begin{equation}
\label{eq:intro-Ophi} \mathcal O_\varphi(X) = \left\{ f\in L^1:
\left\| \frac{f^{**}-f^*}{\varphi} \right\|_X+\|f\|_1<\infty
\right\}.
\end{equation}
Equivalently, this class is determined by the functional
\[
\mathcal N_{\varphi,X}(f) = \left\| \frac{f^{**}-f^*}{\varphi}
\right\|_X+\|f\|_1.
\]

The purpose of this paper is to answer two natural questions.

\medskip

\noindent \emph{What is the least rearrangement-invariant Banach
function space containing \(\mathcal O_\varphi(X)\)?}

\smallskip

\noindent \emph{When is \(\mathcal N_{\varphi,X}\) equivalent to a
rearrangement-invariant Banach function norm?}

\medskip

The main objective of the paper can be stated in one sentence:
normability of the oscillation class is equivalent to the collapse
of the lower and upper Aronszajn--Gagliardo interpolation extremals
associated with \(X\). Moreover, the upper extremal generates the
least r.i. Banach space containing \(\mathcal O_\varphi(X)\).

The two questions formulated above have important precedents in the
literature. Weighted \(L^p\)-models of oscillation classes and their
functional properties were studied in
\cite{CarroGogatishviliMartinPick2005}. A related interpolation
approach to the power-profile oscillation class was developed in
\cite{GogatishviliPickSchneider2012}, where \(K\)-functionals
involving such classes were characterized and applied to the
determination of rearrangement-invariant hulls of Besov spaces. In
the power case \(\varphi(t)=t^{m/n}\), optimal-domain,
optimal-range, and normability problems were developed in
\cite{KermanPick2006,KermanPick2009,Pustylnik2008}. Related
optimality and supremal-operator methods for Sobolev embeddings
associated with general isoperimetric profiles appear in
\cite{Kubicek2026}. Recent developments also include explicit
\(K\)-functional descriptions for classical Lorentz spaces measuring
oscillation and nonstandard Calder\'on-type results for endpoint
Lorentz couples; see
\cite{GogatishviliNevesPickTurcinova2025,KubicekCalderon2025}. Thus
both the normability and optimality problems have substantial
precedents. In particular, the power-profile normability problem has
already been characterized in earlier work. The point here is to
place these questions within a single interpolation framework for
general admissible profiles and to transfer the resulting
order-reflecting structure to the Aronszajn--Gagliardo extremals.

The simultaneous appearance of these two problems is natural in the
context of sharp rearrangement inequalities. Many endpoint
inequalities have the form\footnote{Throughout the paper,
\(A\lesssim B\) means that \(A\leq CB\) for some constant \(C>0\)
independent of the relevant variables, and \(A\simeq B\) means that
both \(A\lesssim B\) and \(B\lesssim A\) hold.},
\begin{equation}
\label{eq:intro-general-profile} f^{**}(t)-f^*(t) \lesssim
\varphi(t)\,\mathcal A_f(t),
\end{equation}
where \(\mathcal A_f\) is an analytic quantity associated with \(f\)
and the profile \(\varphi\) reflects the underlying geometry. For
instance, in the classical Euclidean Sobolev setting one has
\(\varphi(t)=t^{1/n}\). More precisely
\[
f^{**}(t)-f^*(t) \lesssim t^{1/n}|\nabla f|^{**}(t);
\]
see, for example, \cite{MartinMilmanPustylnik2007}. The profile may
encode, for instance, dimensional, isoperimetric, volume-growth, or
capacitary information. It is precisely the profile occurring in
inequalities such as \eqref{eq:intro-general-profile} that motivates
the normalization in \eqref{eq:intro-Ophi}. Classical roots of this
point of view lie in rearrangement and symmetrization methods for
Sobolev inequalities \cite{Mazya1985,Talenti1995}, while its
systematic use in endpoint embeddings, self-improvement, and
symmetrization was developed in a series of works of M.~Milman and
collaborators; see, among others,
\cite{BasteroMilmanRuiz2003,MilmanPustylnik2004,
MartinMilmanPustylnik2007,MartinMilman2010,MartinMilman2014} and the
references therein.

The main contribution of the present paper is a unified
interpolation method for these problems. Besides extending the
power-profile theory to general admissible functions \(\varphi\),
the method places normability, optimal exteriors, and the associated
operator criteria within a single structural framework. The relevant
object is the Lorentz couple
\[
(L^1,\Lambda_\varphi),
\]
and the decisive point is that the nonlinear oscillation
construction retains, in a precise sense, the interpolation geometry
of the underlying r.i. space. This makes it possible to treat the
optimal exterior and the normability problem simultaneously: the
former is governed by the upper Aronszajn--Gagliardo extremal,
whereas the latter occurs precisely when the lower and upper
extremals collapse. Thus the method not only extends earlier results
but also explains them from a common interpolation point of view.

There is also a classical reason why interpolation is particularly
well suited to this setting. Real interpolation methods based
directly on the derivative \(k(t,f)=\frac{d}{dt}K(t,f)\) of the
\(K\)-functional already appear in Bennett's work
\cite{Bennett1974}. Jawerth and Milman later used Gagliardo diagrams
to study the derivatives of the \(K\)- and \(E\)-functionals and the
associated weak-type spaces; see \cite{JawerthMilman1989}. This
point of view was subsequently developed in terms of Gagliardo
coordinate spaces in \cite{Milman2016BMO}.

For the basic couple \((L^1,L^\infty)\),
\[
K(t,f;L^1,L^\infty)=tf^{**}(t),
\]
and therefore
\[
K(t,f)-tK'(t,f) = t\bigl(f^{**}(t)-f^*(t)\bigr).
\]
Thus, up to the factor \(t\), the oscillation is the first Gagliardo
coordinate of the basic couple; see
\cite[Section~3.1]{BerghLofstrom} and
\cite{JawerthMilman1989,Milman2016BMO}. More importantly, the same
phenomenon occurs for the Lorentz couple \((L^1,\Lambda_\varphi)\):
its exact \(K\)-geometry again produces \(f^{**}-f^*\) as the
complementary Gagliardo coordinate; see Section~2.

 We shall therefore work with r.i.
spaces \(X\) satisfying
\begin{equation*}
\Lambda_\varphi\hookrightarrow X\hookrightarrow L^1.
\end{equation*}
With such an \(X\) we associate the Aronszajn--Gagliardo extremal
spaces
\begin{equation*}
 X_-\hookrightarrow X\hookrightarrow X_+,
\end{equation*}
where \(X_-\) is the greatest interpolation space for
\((L^1,\Lambda_\varphi)\) contained in \(X\), while \(X_+\) is the
least interpolation space for this couple containing \(X\); see
\cite{AronszajnGagliardo1965,BennettSharpley,BerghLofstrom,
KreinPetuninSemenov}.

The statement given above can now be made precise. Our main theorem
gives
\begin{equation}
\label{eq:intro-main-equivalence}
\begin{aligned}
\mathcal N_{\varphi,X}\text{ is r.i.-normable}
&\quad\Longleftrightarrow\quad X_-\simeq X_+
\\
&\quad\Longleftrightarrow\quad X\in\Int(L^1,\Lambda_\varphi),
\end{aligned}
\end{equation}
where \(\Int(L^1,\Lambda_\varphi)\) denotes the class of
interpolation spaces for the couple \((L^1,\Lambda_\varphi)\).

A key ingredient of this method is a recovery principle showing that
the oscillation construction reflects the inclusion order of the
underlying rearrangement-invariant spaces. More precisely, for r.i.
spaces \(A\) and \(X\) satisfying
\[
\Lambda_\varphi\hookrightarrow A,X,
\]
we prove
\begin{equation*}
\mathcal O_\varphi(A)\hookrightarrow\mathcal O_\varphi(X)
\quad\Longleftrightarrow\quad A\hookrightarrow X.
\end{equation*}
This order-reflection principle allows the Aronszajn--Gagliardo
extremal construction for \((L^1,\Lambda_\varphi)\) to be
transferred directly to the oscillation classes.

In particular, the least r.i. Banach space containing \(\mathcal
O_\varphi(X)\) is, up to equivalence,
\[
\mathcal O_\varphi(X_+),
\]
while the lower extremal generates the canonical normable interior
\[
\mathcal O_\varphi(X_-) \hookrightarrow \mathcal O_\varphi(X).
\]
Thus the two Aronszajn--Gagliardo extremals provide canonical
normable envelopes from below and above, and
\eqref{eq:intro-main-equivalence} shows that normability occurs
exactly when they coincide.

We also obtain an equivalent dual formulation in terms of an
explicit supremal operator associated with the dual
Lorentz--Marcinkiewicz couple. This turns the abstract collapse
condition \(X_-\simeq X_+\) into a concrete operator criterion.

The criterion becomes particularly rigid at the critical scale. If
the fundamental function of \(X\) satisfies
\[
\phi_X(t)\simeq\varphi(t),
\]
then the collapse in \eqref{eq:intro-main-equivalence} occurs if and
only if
\[
X\simeq\Lambda_\varphi.
\]
Thus, among r.i. spaces having the critical fundamental function
\(\varphi\), the Lorentz endpoint \(\Lambda_\varphi\) is, up to
equivalence of norms, the unique space for which the corresponding
oscillation functional is normable.

The abstract results lead to explicit conclusions on classical
scales. For \(X=L^{p,q}\), \(1<p<\infty\), \(1\leq q\leq\infty\), we
identify the upper extremal and the optimal Banach exterior and show
that normability holds precisely when \(q=1\). The limiting Lorentz
and Orlicz scales display subtler borderline phenomena: normability
may persist in the former, whereas logarithmic perturbations in the
latter remain non-normable at the critical threshold. These examples
show that the criterion detects fine rearrangement-invariant
structure beyond the underlying power exponent.

The paper is organized as follows. Section~\ref{sec:background}
collects the necessary background and describes the relevant
\(K\)-geometry of the couple \((L^1,\Lambda_\varphi)\).
Section~\ref{sec:AG} develops the Aronszajn--Gagliardo extremal
construction and its dual supremal description.
Section~\ref{sec:recovery-exterior} establishes the recovery and
order-reflection principles and identifies the optimal Banach
exterior. Section~\ref{sec:collapse} proves the normability and
interpolation-collapse theorem, together with the critical rigidity
result. The final section applies the abstract theory to Lorentz,
limiting Lorentz, and Orlicz spaces.

%%%%%%%%%%%%%%%%%%%%%%%%%%%%%%%%%%%%%%%%%%%%%%%%%%%%%%%%%%%%%%%%%%%%%%%%%%%%%%%
\section{Background and setup}
\label{sec:background}
%%%%%%%%%%%%%%%%%%%%%%%%%%%%%%%%%%%%%%%%%%%%%%%%%%%%%%%%%%%%%%%%%%%%%%%%%%%%%%%

We briefly recall the notation and standard facts concerning
rearrangement-invariant (r.i.) spaces that will be used throughout
the paper. Unless otherwise stated, all function spaces are
considered over the interval \((0,1)\) endowed with Lebesgue
measure. For further background we refer to
\cite{BennettSharpley,KreinPetuninSemenov,BerghLofstrom}.

Let \((0,1)\) be endowed with Lebesgue measure. We denote by
\(L^0(0,1)\) the space of all measurable functions on \((0,1)\)
which are finite almost everywhere, with the usual identification
modulo equality almost everywhere.

For \(f\in L^0(0,1)\), its decreasing rearrangement is defined by
\[
f^*(s) = \inf\bigl\{ \lambda>0: |\{x\in(0,1):|f(x)|>\lambda\}|\leq s
\bigr\}, \qquad 0<s<1.
\]
Associated with \(f^*\), we consider its maximal average
\[
f^{**}(t)=\frac1t\int_0^t f^*(s)\,ds, \qquad 0<t<1.
\]
Recall also that
\begin{equation}
\label{eq:fss-subadditivity} (f+g)^{**}(t)\leq f^{**}(t)+g^{**}(t),
\qquad 0<t<1.
\end{equation}
A basic property of the oscillation is that
\begin{equation}
\label{eq:t-oscillation-increasing} t\bigl(f^{**}(t)-f^*(t)\bigr)
\quad\text{is nondecreasing on }(0,1);
\end{equation}
see, for example, \cite[proof of
Lemma~2.1]{CarroGogatishviliMartinPick2005}. Moreover, if \(f\in
L^1(0,1)\), then
\begin{equation*}
\frac{d}{dt}f^{**}(t) = -\frac{f^{**}(t)-f^*(t)}{t}
\end{equation*}
for almost every \(t\in(0,1)\).

A Banach function space \(X\) on \((0,1)\) is called
rearrangement-invariant (r.i.) if
\[
f^*=g^* \quad\Longrightarrow\quad \|f\|_X=\|g\|_X.
\]
Throughout the paper, Banach function spaces are understood in the
sense of \cite{BennettSharpley}; in particular, the Fatou property
is part of the definition.

The associate space \(X'\) is defined by
\[
\|g\|_{X'} = \sup_{\|f\|_X\leq1} \int_0^1 |f(s)g(s)|\,ds.
\]
It is again an r.i. space and
\begin{equation*}
\|g\|_{X'} = \sup_{\|f\|_X\leq1} \int_0^1 f^*(t)g^*(t)\,dt.
\end{equation*}
Consequently,
\begin{equation*}
\int_0^1 |f(t)g(t)|\,dt \leq \|f\|_X\|g\|_{X'}.
\end{equation*}
Moreover,
\[
X=X''
\]
with equality of norms.

On \((0,1)\), every r.i. space satisfies
\begin{equation*}
L^\infty\hookrightarrow X\hookrightarrow L^1.
\end{equation*}
The fundamental function of \(X\) is defined by
\[
\phi_X(t) = \|\chi_{(0,t)}\|_X, \qquad 0<t<1.
\]
It is increasing and quasi-concave, and the fundamental functions of
\(X\) and \(X'\) satisfy
\begin{equation*}
 \phi_X(t)\phi_{X'}(t)=t,
\qquad 0<t<1.
\end{equation*}
For example, \( \phi_{L^p}(t)=t^{1/p}, \; 1\leq p<\infty\), whereas
\( \phi_{L^\infty}(t)=1.\)

We now introduce the class of profiles which determines the
interpolation couples considered throughout the paper.

\begin{definition}
\label{def:admissible-profile} An increasing concave function
\(\varphi:[0,1]\to[0,1]\) is called \emph{admissible} if
\(\varphi(0)=0\), \(\varphi(1)=1\), and
\begin{equation}
\label{eq:phi-Hardy} A_\varphi := \sup_{0<t<1} \frac1{\varphi(t)}
\int_0^t\frac{\varphi(s)}{s}\,ds <\infty, \qquad B_\varphi :=
\sup_{0<t<1} \varphi(t) \int_t^1\frac{ds}{s\varphi(s)} <\infty.
\end{equation}
\end{definition}

For an admissible \(\varphi\), set
\begin{equation*}
\psi(t)=\frac{t}{\varphi(t)}, \qquad 0<t\leq1.
\end{equation*}
Then \(\psi\) is quasi-concave: it is nondecreasing and
\(\psi(t)/t=1/\varphi(t)\) is nonincreasing.

The Lorentz space associated with \(\varphi\) is defined by
\begin{equation*}
\|f\|_{\Lambda_\varphi} = \int_0^1 f^*(t)\,d\varphi(t).
\end{equation*}
It is an r.i. Banach function space with fundamental function
\[
\phi_{\Lambda_\varphi}(t)=\varphi(t).
\]

The corresponding Marcinkiewicz space is defined by
\begin{equation*}
\|h\|_{M_\psi} = \sup_{0<t<1}\psi(t)h^{**}(t).
\end{equation*}
The classical Lorentz--Marcinkiewicz duality gives
\begin{equation*}
(\Lambda_\varphi)'=M_\psi, \qquad \psi(t)=\frac{t}{\varphi(t)}.
\end{equation*}

By \cite[Section~5, Lemma~5.3 and Theorem~5.3]
{KreinPetuninSemenov}, the first condition in
Definition~\ref{def:admissible-profile} allows \(h^{**}\) to be
replaced by \(h^*\) in the Marcinkiewicz norm. More precisely,
\begin{equation*}
\sup_{0<t<1}\psi(t)h^*(t) \leq \|h\|_{M_\psi} \leq
A_\varphi\sup_{0<t<1}\psi(t)h^*(t).
\end{equation*}

Our main setting will be
\begin{equation}
\label{eq:X-intermediate-general} \Lambda_\varphi\hookrightarrow
X\hookrightarrow L^1,
\end{equation}
where \(X\) is an r.i. space. The second embedding in
\eqref{eq:X-intermediate-general} is automatic on \((0,1)\); we
display it in order to emphasize that \(X\) is an intermediate space
for the ordered couple
\[
(L^1,\Lambda_\varphi).
\]

For \(f\in L^1(0,1)\), define the normalized oscillation
\begin{equation*}
\omega_{\varphi,f}(t) = \frac{f^{**}(t)-f^*(t)}{\varphi(t)}, \qquad
0<t<1.
\end{equation*}
Given \(X\) satisfying \eqref{eq:X-intermediate-general}, set
\begin{equation*}
 \mathcal N_{\varphi,X}(f) =
\|\omega_{\varphi,f}\|_X+\|f\|_1
\end{equation*}
and
\begin{equation*}
\mathcal O_\varphi(X) = \bigl\{ f\in L^1(0,1): \mathcal
N_{\varphi,X}(f)<\infty \bigr\}.
\end{equation*}
We say that \(\mathcal N_{\varphi,X}\) is \emph{r.i.-normable} if
there exists an r.i. space \(Y\) such that
\[
Y=\mathcal O_\varphi(X)
\]
as sets and
\[
\|f\|_Y\simeq\mathcal N_{\varphi,X}(f), \qquad f\in\mathcal
O_\varphi(X).
\]
No normability of \(\mathcal N_{\varphi,X}\) is assumed at this
stage.

Finally, by \eqref{eq:t-oscillation-increasing},
\begin{equation*}
t\varphi(t)\omega_{\varphi,f}(t) = t\bigl(f^{**}(t)-f^*(t)\bigr)
\end{equation*}
is nondecreasing on \((0,1)\).

%%%%%%%%%%%%%%%%%%%%%%%%%%%%%%%%%%%%%%%%%%%%%%%%%%%%%%%%%%%%%%%%%%%%%%%%%%%%%%%
\subsection{Gagliardo geometry of the Lorentz couple}
%%%%%%%%%%%%%%%%%%%%%%%%%%%%%%%%%%%%%%%%%%%%%%%%%%%%%%%%%%%%%%%%%%%%%%%%%%%%%%%

We now make explicit the \(K\)-geometry of the couple \(
(L^1,\Lambda_\varphi). \) The classical sum formula for Lorentz
spaces gives
\begin{equation*}
K(t,f;L^1,\Lambda_\varphi) =
\|f\|_{\Lambda_{\min\{s,t\varphi(s)\}}}, \qquad t>0;
\end{equation*}
see, for example, \cite[Chapter~II]{KreinPetuninSemenov}.

Taking
\[
t=\psi(r)=\frac{r}{\varphi(r)}, \qquad 0<r<1,
\]
the monotonicity of \(\psi\) gives
\[
\min\{s,\psi(r)\varphi(s)\} =
\begin{cases}
s, & 0<s\leq r,\\[1mm]
\psi(r)\varphi(s), & r\leq s<1.
\end{cases}
\]
and therefore
\begin{equation}
\label{eq:prim-K-order} K\bigl(\psi(r),f;L^1,\Lambda_\varphi\bigr) =
\int_0^r f^*(s)\,ds + \psi(r)\int_r^1 f^*(s)\,d\varphi(s).
\end{equation}

The corresponding Gagliardo point can be obtained directly. Fix
\(0<r<1\), put \(\lambda=f^*(r)\), and consider the decomposition
\[
f=f_{0,r}+f_{1,r},
\]
where
\[
f_{0,r}=\operatorname{sgn}(f)(|f|-\lambda)_+, \qquad
f_{1,r}=f-f_{0,r}.
\]
The standard rearrangement identities for truncations give
\[
\|f_{0,r}\|_1 = \int_0^r f^*(s)\,ds-rf^*(r) =
r\bigl(f^{**}(r)-f^*(r)\bigr),
\]
and
\[
\|f_{1,r}\|_{\Lambda_\varphi} = \varphi(r)f^*(r) + \int_r^1
f^*(s)\,d\varphi(s).
\]
Hence, for every \(t>0\),
\[
K(t,f;L^1,\Lambda_\varphi) \leq r\bigl(f^{**}(r)-f^*(r)\bigr) +
t\left( \varphi(r)f^*(r) + \int_r^1 f^*(s)\,d\varphi(s) \right).
\]
Taking \(t=\psi(r)\), using \(r=\varphi(r)\psi(r)\) and
\eqref{eq:prim-K-order}, equality holds. Therefore
\begin{equation*}
\left( r\bigl(f^{**}(r)-f^*(r)\bigr), \, \varphi(r)f^*(r) + \int_r^1
f^*(s)\,d\varphi(s) \right)
\end{equation*}
is a point of contact of the Gagliardo diagram with the supporting
line corresponding to \(t=\psi(r)\). In particular,
\(r(f^{**}(r)-f^*(r))\) is its first Gagliardo coordinate.

This formulation does not require \(\psi\) to be strictly
increasing. If \(\psi\) is constant on an interval, the supporting
line may have more than one point of contact, and the preceding
construction provides the point corresponding to each \(r\).

%%%%%%%%%%%%%%%%%%%%%%%%%%%%%%%%%%%%%%%%%%%%%%%%%%%%%%%%%%%%%%%%%%%%%%%%%%%%%%%
\section{Aronszajn--Gagliardo extremals and the optimal Banach exterior}
\label{sec:AG}
%%%%%%%%%%%%%%%%%%%%%%%%%%%%%%%%%%%%%%%%%%%%%%%%%%%%%%%%%%%%%%%%%%%%%%%%%%%%%%%

We shall use the classical extremal interpolation construction of
Aronszajn and Gagliardo. Since the abstract theory is standard, we
record only the facts needed below; see
\cite{AronszajnGagliardo1965,BerghLofstrom,BrudnyiKrugljak}.

Throughout this section, let \(X\) be an r.i. space  satisfying
\begin{equation*}
\Lambda_\varphi\hookrightarrow X\hookrightarrow L^1.
\end{equation*}
By K\"othe duality,
\[
L^\infty\hookrightarrow X'\hookrightarrow M_\psi.
\]

We shall write
\begin{equation*}
 \vec L_\varphi=(L^1,\Lambda_\varphi), \qquad
\vec B_\varphi=(L^\infty,M_\psi)
\end{equation*}
for these mutually associate couples.

For either of the couples
\[
\vec A=(A_0,A_1)\in\{\vec L_\varphi,\vec B_\varphi\},
\]
let
\[
\mathfrak T(\vec A) = \left\{ T\text{ linear}: \|T\|_{A_0\to
A_0}\leq1,\quad \|T\|_{A_1\to A_1}\leq1 \right\}.
\]
If \(E\) is an intermediate r.i. space of \(\vec A\), its lower
Aronszajn--Gagliardo extremal is defined by
\begin{equation*}
\|f\|_{E_-} = \sup_{T\in\mathfrak T(\vec A)}\|Tf\|_E.
\end{equation*}

Let
\begin{equation*}
 \mathfrak T_\varphi = \left\{ T\text{ linear}:
\|T\|_{L^1\to L^1}\leq1, \quad
\|T\|_{\Lambda_\varphi\to\Lambda_\varphi}\leq1 \right\}.
\end{equation*}
be the semigroup of exact endomorphisms of \(\vec L_\varphi\).

For an intermediate r.i. space \(E\), relative to either of the
associate couples, we use the canonical K\"othe-dual representative
of the upper extremal,
\begin{equation*}
 E_+:=((E')_-)'.
\end{equation*}
In particular,
\begin{equation}
\label{eq:Xplus-phi} X_+:=((X')_-)',
\end{equation}
where \((X')_-\) is the lower extremal of \(X'\) relative to \(\vec
B_\varphi\).
\begin{remark}
The dual formula \eqref{eq:Xplus-phi} is the most convenient one for
our purposes, but it is equivalent to the classical projective, or
sum, construction of Aronszajn--Gagliardo. Namely, one may form the
space of sums
\[
f=\sum_{j=1}^\infty T_jf_j, \qquad T_j\in\mathfrak T_\varphi, \quad
f_j\in X,
\]
with projective norm
\[
\|f\|_{X_+^{\rm proj}} = \inf \left\{ \sum_{j=1}^\infty\|f_j\|_X:
f=\sum_{j=1}^\infty T_jf_j \right\}.
\]
The corresponding r.i. representative with the Fatou property is
equivalent to \(X_+\). Thus \eqref{eq:Xplus-phi} may be regarded as
the dual realization of the classical sum method; see
\cite[Theorem~2.5.1]{BerghLofstrom} and \cite[Chapter~2,
Section~2.5]{BrudnyiKrugljak}.
\end{remark}

\begin{theorem}[Aronszajn--Gagliardo extremals]
\label{thm:phi-extremals} The spaces \(X_-\) and \(X_+\) are r.i.
interpolation spaces of \(\vec L_\varphi\). Moreover,
\begin{equation*}
\Lambda_\varphi \hookrightarrow X_- \hookrightarrow X
\hookrightarrow X_+ \hookrightarrow L^1.
\end{equation*}
The space \(X_-\) is the greatest r.i. interpolation space contained
in \(X\), while \(X_+\) is the least r.i. interpolation space
containing \(X\). With respect to the mutually associate couples
\(\vec L_\varphi\) and \(\vec B_\varphi\),
\begin{equation}
\label{eq:phi-extremal-duality} (X_-)'\simeq(X')_+, \qquad
(X_+)'\simeq(X')_-.
\end{equation}
Finally,
\begin{equation}
\label{eq:phi-dual-interpolation} X\in\Int(L^1,\Lambda_\varphi)
\quad\Longleftrightarrow\quad X'\in\Int(L^\infty,M_\psi),
\end{equation}
and
\begin{equation*}
\begin{aligned}
X\in\Int(L^1,\Lambda_\varphi) &\quad\Longleftrightarrow\quad
X_-\simeq X\simeq X_+ \\
&\quad\Longleftrightarrow\quad X_-\simeq X_+.
\end{aligned}
\end{equation*}
\end{theorem}
\begin{proof}
The extremal construction, including the projective realization and
the greatest--least properties, is the classical
Aronszajn--Gagliardo theorem; see
\cite[Theorem~2.5.1]{BerghLofstrom} and \cite[Chapter~2,
Section~2.5]{BrudnyiKrugljak}. K\"othe duality for the mutually
associate couples, together with the Fatou identity \(E''=E\), gives
\[
(X_-)' \simeq (X')_+, \qquad (X_+)' \simeq (X')_-;
\]
see also \cite{BennettSharpley}. The remaining assertions follow
from these extremal properties and K\"othe duality.
\end{proof}

The lower extremal on the associate side admits an explicit supremal
realization. Define
\begin{equation*}
 S_\psi h(t) = \frac1{\psi(t)} \sup_{0<s\leq
t}\psi(s)h^*(s), \qquad 0<t<1.
\end{equation*}

\begin{proposition}
\label{prop:dual-extremal} Let \(E\) be an intermediate r.i. space
of \(\vec B_\varphi\). Then\footnote{Here and below, a subscript in
\(\simeq\) indicates the parameters on which the implicit constants
may depend.}
\begin{equation}
\label{eq:K-Spsi} K\left(\frac1{\psi(r)},h;L^\infty,M_\psi\right)
\simeq_\varphi S_\psi h(r), \qquad 0<r<1,
\end{equation}
and
\begin{equation}
\label{eq:K-S-equivalence} K(t,S_\psi h;\vec B_\varphi)
\simeq_\varphi K(t,h;\vec B_\varphi), \qquad t>0.
\end{equation}
Consequently,
\begin{equation}
\label{eq:S-interpolation} E\in\Int(L^\infty,M_\psi)
\quad\Longleftrightarrow\quad S_\psi:E\longrightarrow E \text{ is
bounded},
\end{equation}
and the lower extremal of \(E\) relative to \(\vec B_\varphi\) is
\begin{equation}
\label{eq:dual-core-general} E_- = \{h:S_\psi h\in E\}, \qquad
\|h\|_{E_-} \simeq_{\varphi,E} \|S_\psi h\|_E.
\end{equation}
\end{proposition}

\begin{proof}
The classical \(K\)-functional formula for a Marcinkiewicz space and
\(L^\infty\), together with the admissibility condition
\eqref{eq:phi-Hardy}, gives
\[
K(u,h;M_\psi,L^\infty) \simeq_\varphi
\sup_{0<s<1}\min\{\psi(s),u\}\,h^*(s);
\]
see \cite{CwikelNilsson1985}. Reversing the couple gives
\[
K(t,h;L^\infty,M_\psi) = tK(t^{-1},h;M_\psi,L^\infty).
\]
Hence, for \(0<r<1\),
\[
\begin{aligned}
K\left(\frac1{\psi(r)},h;L^\infty,M_\psi\right) &\simeq_\varphi
\frac1{\psi(r)} \sup_{0<s<1}\min\{\psi(s),\psi(r)\}h^*(s)
\\
&= \frac1{\psi(r)} \sup_{0<s\leq r}\psi(s)h^*(s) = S_\psi h(r).
\end{aligned}
\]
Here, in passing to the supremum over \(0<s\leq r\), we have used
that, for \(s>r\),
\[
\min\{\psi(s),\psi(r)\}h^*(s) = \psi(r)h^*(s) \leq \psi(r)h^*(r),
\]
since \(\psi\) is nondecreasing and \(h^*\) is nonincreasing. This
proves \eqref{eq:K-Spsi}.

We next prove \eqref{eq:K-S-equivalence}. First observe that
\(S_\psi h\) is nonincreasing. Indeed, let \(0<t_1<t_2<1\). If
\(s\leq t_1\), then
\[
\frac{\psi(s)h^*(s)}{\psi(t_2)} \leq \frac{\psi(s)h^*(s)}{\psi(t_1)}
\leq S_\psi h(t_1),
\]
whereas, if \(t_1<s\leq t_2\), then
\[
\frac{\psi(s)h^*(s)}{\psi(t_2)} \leq h^*(s) \leq h^*(t_1) \leq
S_\psi h(t_1).
\]
Taking the supremum over \(0<s\leq t_2\) gives
\[
S_\psi h(t_2)\leq S_\psi h(t_1).
\]
Moreover, directly from the definition,
\[
S_\psi h\geq h^*.
\]

Now fix \(u>0\). If \(s\leq t\), then
\[
\min\left\{1,\frac{u}{\psi(t)}\right\}\psi(s) \leq
\min\{\psi(s),u\}.
\]
Therefore
\[
\begin{aligned}
\min\{\psi(t),u\}S_\psi h(t) &=
\min\left\{1,\frac{u}{\psi(t)}\right\} \sup_{0<s\leq t}\psi(s)h^*(s)
\\
&\leq \sup_{0<s\leq t} \min\{\psi(s),u\}h^*(s).
\end{aligned}
\]
Taking the supremum over \(t\), and using \(S_\psi h\geq h^*\) for
the reverse inequality, we obtain
\[
\sup_{0<t<1}\min\{\psi(t),u\}S_\psi h(t) =
\sup_{0<t<1}\min\{\psi(t),u\}h^*(t).
\]
Hence \eqref{eq:K-S-equivalence} follows from the same
\(K\)-functional formula. By \cite[Theorem~9]{CerdaCollMartin2005},
a Banach function space \(F\) satisfies that \((F,L^\infty)\) is a
universal right Calder\'on couple if and only if \(F=M(F)\). Since
\(M_\psi\) is a Marcinkiewicz space, \(M(M_\psi)=M_\psi\).
Therefore, by the standard characterization of interpolation spaces
for Calder\'on couples in terms of \(K\)-monotonicity; see, for
example, \cite{BennettSharpley,BrudnyiKrugljak},
\eqref{eq:S-interpolation} follows.

Finally, if \(T\) is an exact endomorphism of \(\vec B_\varphi\),
then \(K(t,Th)\leq K(t,h)\), and \eqref{eq:K-Spsi} gives
\[
S_\psi(Th)\lesssim_\varphi S_\psi h.
\]
Since
\[
(Th)^*\leq S_\psi(Th),
\]
we obtain
\[
\|Th\|_E \leq \|S_\psi(Th)\|_E \lesssim_\varphi \|S_\psi h\|_E.
\]
Taking the supremum over exact endomorphisms yields
\[
\|h\|_{E_-} \lesssim_\varphi \|S_\psi h\|_E.
\]
Conversely, \(E_-\in\Int(\vec B_\varphi)\) by the extremal theorem,
so \eqref{eq:S-interpolation} gives
\[
S_\psi:E_-\longrightarrow E_-.
\]
Since \(E_-\hookrightarrow E\),
\[
\|S_\psi h\|_E \lesssim \|S_\psi h\|_{E_-} \lesssim \|h\|_{E_-}.
\]
This proves \eqref{eq:dual-core-general}.
\end{proof}

Combining Theorem~\ref{thm:phi-extremals} and
Proposition~\ref{prop:dual-extremal}, we shall repeatedly use
\begin{equation}
\label{eq:pre-collapse} X\in\Int(L^1,\Lambda_\varphi)
\quad\Longleftrightarrow\quad S_\psi:X'\to X'
\quad\Longleftrightarrow\quad X_-\simeq X_+.
\end{equation}
Moreover, by \eqref{eq:phi-extremal-duality} and
\eqref{eq:dual-core-general},
\begin{equation*}
 X_+ \simeq \bigl( \{h:S_\psi h\in
X'\}, \ \|S_\psi h\|_{X'} \bigr)'.
\end{equation*}
Thus the dual realization of the upper Aronszajn--Gagliardo extremal
is completely explicit.

\begin{remark}
For the power profile
\[
\varphi(t)=t^{m/n}, \qquad \psi(t)=t^{1-m/n},
\]
the equivalence
\[
X\in\Int(L^1,L^{n/m,1}) \quad\Longleftrightarrow\quad
S_\psi:X'\longrightarrow X'
\]
appears in \cite[Theorem~2.2]{KermanPick2009}, where it is obtained
by a different argument based on endpoint estimates,
quasisubadditivity of the corresponding supremal operator, and
boundedness of dilations.

More recently, supremal operators of the same type have been used in
\cite{Kubicek2026} to characterize optimal domain and target spaces
in Sobolev embeddings associated with general isoperimetric
profiles. In the notation of that work, taking \(I=\psi=t/\varphi\)
gives
\[
\frac{I(t)}{t}\bigl(f^{**}(t)-f^*(t)\bigr) =
\frac{f^{**}(t)-f^*(t)}{\varphi(t)}.
\]
The present approach places these phenomena in the abstract
interpolation structure of the Lorentz couple
\((L^1,\Lambda_\varphi)\), where the same supremal condition is
identified with the collapse of the Aronszajn--Gagliardo extremals.
\end{remark}

%%%%%%%%%%%%%%%%%%%%%%%%%%%%%%%%%%%%%%%%%%%%%%%%%%%%%%%%%%%%%%%%%%%%%%%%%%%%%%%
\section{Recovery and the optimal Banach exterior}
\label{sec:recovery-exterior}
%%%%%%%%%%%%%%%%%%%%%%%%%%%%%%%%%%%%%%%%%%%%%%%%%%%%%%%%%%%%%%%%%%%%%%%%%%%%%%%

We now connect the interpolation structure obtained above with the
oscillation class. Define
\begin{equation*}
Q_\varphi g(t) = \int_t^1\frac{\varphi(s)}{s}g(s)\,ds.
\end{equation*}
The terminology ``recovery operator'' comes from the identity
\begin{equation}
\label{eq:fstarstar-recovery-general} f^{**}(t) = \|f\|_1 +
Q_\varphi\omega_{\varphi,f}(t), \qquad 0<t<1.
\end{equation}
Indeed, since
\[
\frac{d}{dt}f^{**}(t) = -\frac{\varphi(t)}{t}\omega_{\varphi,f}(t)
\]
for almost every \(t\in(0,1)\), integration from \(t\) to \(1\),
together with \(f^{**}(1)=\|f\|_1\), gives
\eqref{eq:fstarstar-recovery-general}.

To recover decreasing data we shall also use the operator
\begin{equation*}
P_\varphi g(t) = \frac1{t\varphi(t)} \int_0^t\varphi(s)g(s)\,ds.
\end{equation*}

\begin{lemma}
\label{lem:P-phi} The operator \(P_\varphi\) is bounded on \(L^1\)
and \(L^\infty\), with
\[
\|P_\varphi\|_{L^\infty\to L^\infty}\leq1, \qquad
\|P_\varphi\|_{L^1\to L^1}\leq B_\varphi.
\]
Consequently,
\begin{equation}
\label{eq:P-X} P_\varphi:X\longrightarrow X
\end{equation}
for every r.i. space \(X\). Moreover, if \(g=g^*\geq0\), then
\begin{equation}
\label{eq:P-lower} P_\varphi g(t)\geq\frac12g(t), \qquad 0<t<1.
\end{equation}
\end{lemma}

\begin{proof}
Since \(\varphi\) is increasing,
\[
P_\varphi g(t) \leq \|g\|_\infty \frac1{t\varphi(t)}
\int_0^t\varphi(s)\,ds \leq \|g\|_\infty.
\]
Thus \(P_\varphi:L^\infty\to L^\infty\). For \(g\geq0\), Fubini's
theorem and \eqref{eq:phi-Hardy} give
\[
\begin{aligned}
\int_0^1P_\varphi g(t)\,dt &= \int_0^1g(s)\varphi(s)
\int_s^1\frac{dt}{t\varphi(t)}\,ds
\\
&\leq B_\varphi\int_0^1g(s)\,ds.
\end{aligned}
\]
Hence \(P_\varphi:L^1\to L^1\), and \eqref{eq:P-X} follows by
interpolation.

Finally, if \(g=g^*\geq0\), then
\[
\int_0^t\varphi(s)g(s)\,ds \geq g(t)\int_0^t\varphi(s)\,ds.
\]
By concavity and \(\varphi(0)=0\),
\[
\varphi(s)\geq\frac{s}{t}\varphi(t), \qquad 0<s<t,
\]
and therefore
\[
\int_0^t\varphi(s)\,ds \geq \frac{t\varphi(t)}2.
\]
This proves \eqref{eq:P-lower}.
\end{proof}

The preceding lemma gives the recovery principle that will be used
throughout the rest of the paper.

\begin{theorem}[Recovery and order reflection]
\label{thm:recovery-general} Let \(X\) be an intermediate r.i. space
for \(\vec L_\varphi=(L^1,\Lambda_\varphi)\). Then, for every
\(g=g^*\geq0\) in \(X\),
\begin{equation}
\label{eq:recovery-pointwise} \omega_{\varphi,Q_\varphi g}=P_\varphi
g.
\end{equation}
Consequently,
\begin{equation}
\label{eq:recovery-equivalence} \mathcal N_{\varphi,X}(Q_\varphi g)
\simeq_{\varphi,X}\|g\|_X.
\end{equation}
Moreover, if \(A\) is another intermediate r.i. space for \(\vec
L_\varphi\), then
\begin{equation}
\label{eq:order-reflection-general} \mathcal
O_\varphi(A)\hookrightarrow\mathcal O_\varphi(X)
\quad\Longleftrightarrow\quad A\hookrightarrow X.
\end{equation}
\end{theorem}

\begin{proof}
Let \(g=g^*\geq0\) and set \(u=Q_\varphi g\). Since \(u\) is
nonincreasing, \(u^*=u\). By Fubini's theorem,
\[
\begin{aligned}
\int_0^t u(s)\,ds &= \int_0^t
\int_s^1\frac{\varphi(r)}{r}g(r)\,dr\,ds
\\
&= \int_0^t\varphi(r)g(r)\,dr +
t\int_t^1\frac{\varphi(r)}{r}g(r)\,dr.
\end{aligned}
\]
Hence
\[
u^{**}(t)-u(t) = \frac1t\int_0^t\varphi(r)g(r)\,dr,
\]
and division by \(\varphi(t)\) gives \eqref{eq:recovery-pointwise}.

By Lemma~\ref{lem:P-phi},
\[
\frac12\|g\|_X \leq \|P_\varphi g\|_X \lesssim_\varphi \|g\|_X.
\]
Moreover,
\[
\begin{aligned}
\|Q_\varphi g\|_1 &= \int_0^1
\int_t^1\frac{\varphi(s)}{s}g(s)\,ds\,dt
\\
&= \int_0^1\varphi(s)g(s)\,ds \leq \|g\|_1 \lesssim_X \|g\|_X.
\end{aligned}
\]
Together with \eqref{eq:recovery-pointwise}, this proves
\eqref{eq:recovery-equivalence}.

If \(A\hookrightarrow X\), then
\[
\mathcal O_\varphi(A)\hookrightarrow\mathcal O_\varphi(X)
\]
follows directly from the definition. Conversely, assume this
embedding and let \(g=g^*\geq0\) belong to \(A\). Then
\[
\begin{aligned}
\|g\|_X &\lesssim \mathcal N_{\varphi,X}(Q_\varphi g)
\\
&\lesssim \mathcal N_{\varphi,A}(Q_\varphi g)
\\
&\lesssim \|g\|_A.
\end{aligned}
\]
Applying this estimate to \(g=f^*\) and using rearrangement
invariance gives
\[
\|f\|_X\lesssim\|f\|_A.
\]
Thus \(A\hookrightarrow X\), proving
\eqref{eq:order-reflection-general}.
\end{proof}

%%%%%%%%%%%%%%%%%%%%%%%%%%%%%%%%%%%%%%%%%%%%%%%%%%%%%%%%%%%%%%%%%%%%%%%%%%%%%%%
\subsection{The optimal Banach exterior}
%%%%%%%%%%%%%%%%%%%%%%%%%%%%%%%%%%%%%%%%%%%%%%%%%%%%%%%%%%%%%%%%%%%%%%%%%%%%%%%

The adjoint of \(Q_\varphi\) with respect to the integral pairing is
\begin{equation*}
 R_\varphi h(t) = \frac{\varphi(t)}{t}\int_0^t
h(s)\,ds = \frac1{\psi(t)}\int_0^t h(s)\,ds.
\end{equation*}
Indeed, Tonelli's theorem gives, for nonnegative measurable \(u\)
and \(h\),
\begin{equation}
\label{eq:Q-R-general} \int_0^1Q_\varphi u(t)h(t)\,dt =
\int_0^1u(t)R_\varphi h(t)\,dt.
\end{equation}

Define
\[
Y_X = \left\{ h:\|\varphi(t)h^{**}(t)\|_{X'}<\infty \right\},
\]
with
\begin{equation}
\label{eq:Y-X-general} \|h\|_{Y_X} = \|R_\varphi h^*\|_{X'} =
\|\varphi(t)h^{**}(t)\|_{X'}.
\end{equation}
By the standard construction of symmetric spaces defined through
\(h^{**}\), \(Y_X\) is an r.i. space; see
\cite[p.~125]{KreinPetuninSemenov}.

In the power case \(\varphi(t)=t^\alpha\), one has
\[
R_\varphi h^*(t)=t^\alpha h^{**}(t),
\]
and the space \(Y_X'\) recovers the optimal r.i. range associated
with the Copson operator \(Q_\alpha\) in
\cite{KermanPick2006,KermanPick2009}. The argument below recovers
this optimal range, and its minimality, directly from the present
duality and recovery framework.

As usual, minimality of an r.i. space will always be understood with
respect to continuous embeddings.

\begin{theorem}[The optimal Banach exterior]
\label{thm:optimal-exterior} Let \(X\) be an intermediate r.i. space
for \(\vec L_\varphi\). Then \(Y_X'\) is the least r.i. space
containing \(\mathcal O_\varphi(X)\), and
\begin{equation*}
Y_X' \simeq \mathcal O_\varphi(X_+).
\end{equation*}
Equivalently, the optimal Banach exterior is obtained through the
factorization
\begin{equation*}
X \longmapsto X_+ \longmapsto \mathcal O_\varphi(X_+) \simeq Y_X'.
\end{equation*}
\end{theorem}

\begin{proof}
We first record the normability fact needed below. Let
\[
E\in\Int(L^1,\Lambda_\varphi), \qquad u=\omega_{\varphi,f},
\]
and set
\[
D_E(u) = \sup_{\substack{g=g^*\geq0\\ \|g\|_{E'}\leq1}}
\int_0^1u(t)g(t)\,dt.
\]
By H\"older's inequality,
\[
D_E(u)\leq\|u\|_E.
\]

Recall that
\[
t\varphi(t)u(t) = t\bigl(f^{**}(t)-f^*(t)\bigr)
\]
is nondecreasing. Hence, if \(0<t<1/2\) and \(t\leq s\leq2t\), then,
using \(s\leq2t\) and \(\varphi(s)\leq\varphi(2t)\leq2\varphi(t)\),
\[
u(s)\geq\frac14u(t).
\]
Thus, for
\[
\mathcal Qu(t)=\int_t^1u(s)\,\frac{ds}{s},
\]
\[
\mathcal Qu(t) \geq \frac{\log2}{4}\,u(t), \qquad 0<t<1/2.
\]
For \(1/2\leq t<1\),
\[
u(t) \leq \frac{f^{**}(t)}{\varphi(t)} \lesssim \|f\|_1.
\]
Consequently,
\begin{equation}
\label{eq:u-Copson-norm} \|u\|_E \lesssim_E \|\mathcal
Qu\|_E+\|f\|_1.
\end{equation}

The operator \(\mathcal Q\) is bounded on \(L^1\), and its
boundedness on \(\Lambda_\varphi\) follows from
\(A_\varphi<\infty\). Hence \(\mathcal Q:E\to E\), and its adjoint
Hardy operator
\[
\mathcal Pg(t) = \frac1t\int_0^tg(s)\,ds
\]
is bounded on \(E'\). Since \(\mathcal Qu\) and \(\mathcal Pg\) are
nonincreasing for \(g=g^*\geq0\), K\"othe duality and Fubini's
theorem give
\[
\|\mathcal Qu\|_E = \sup_{\substack{g=g^*\geq0\\ \|g\|_{E'}\leq1}}
\int_0^1u(t)\mathcal Pg(t)\,dt \lesssim_E D_E(u).
\]
Together with \eqref{eq:u-Copson-norm}, this yields
\begin{equation}
\label{eq:D-equivalence-main} \|u\|_E+\|f\|_1 \simeq_{\varphi,E}
D_E(u)+\|f\|_1.
\end{equation}

Since \(E\in\Int(L^1,\Lambda_\varphi)\), K\"othe duality and
Proposition~\ref{prop:dual-extremal} give
\[
E'\in\Int(L^\infty,M_\psi), \qquad S_\psi:E'\longrightarrow E'.
\]
Recall that
\[
S_\psi g(t) = \frac1{\psi(t)} \sup_{0<s\leq t}\psi(s)g^*(s).
\]
For \(g=g^*\geq0\), \(S_\psi g\) is nonincreasing, \(S_\psi g\geq
g\), and \(\psi S_\psi g\) is nondecreasing. Thus, with
\[
\mathcal C_\psi = \left\{ g=g^*\geq0: \psi(t)g(t)\text{ is
nondecreasing} \right\},
\]
the boundedness of \(S_\psi\) on \(E'\) gives
\begin{equation}
\label{eq:D-cone} D_E(u) \simeq_{\varphi,E}
\sup_{\substack{g\in\mathcal C_\psi\\ \|g\|_{E'}\leq1}}
\int_0^1u(t)g(t)\,dt.
\end{equation}

Fix \(g\in\mathcal C_\psi\), write \(g(1)=g(1-)\), and set
\[
A_g(t)=\psi(t)g(t).
\]
Then \(A_g\) is nondecreasing and
\[
u(t)g(t) = -A_g(t)\bigl(f^{**}\bigr)'(t).
\]
Since \(f^{**}\) is locally absolutely continuous, Stieltjes
integration by parts on \([a,1]\) gives, for \(0<a<1\),
\begin{equation}
\label{eq:integration-parts-cutoff} \int_a^1u(t)g(t)\,dt+g(1)\|f\|_1
= A_g(a)f^{**}(a) + \int_{(a,1]}f^{**}(t)\,dA_g(t).
\end{equation}
As \(a\downarrow0\), the left-hand side increases to
\[
\int_0^1u(t)g(t)\,dt+g(1)\|f\|_1.
\]
Hence, by \eqref{eq:integration-parts-cutoff} and monotone
convergence,
\[
\sup_{0<a<1} \left[ A_g(a)f^{**}(a) + \int_{(a,1]}f^{**}(t)\,dA_g(t)
\right] = \int_0^1u(t)g(t)\,dt+g(1)\|f\|_1.
\]
We define
\[
\|f\|_{\widetilde E} = \|f\|_1 + \sup_{\substack{g\in\mathcal
C_\psi\\ \|g\|_{E'}\leq1}} \sup_{0<a<1} \left[ A_g(a)f^{**}(a) +
\int_{(a,1]}f^{**}(t)\,dA_g(t) \right].
\]
This defines an r.i. Banach function norm. Rearrangement invariance
and lattice monotonicity are immediate, while the triangle
inequality follows from \eqref{eq:fss-subadditivity}. The
\(L^1\)-term gives the required local integrability. Moreover, if
\(|A|=a\), then
\[
\omega_{\varphi,\chi_A}(t) = \frac{a}{t\varphi(t)} \chi_{(a,1)}(t)
\in L^\infty\subset E,
\]
so \eqref{eq:D-cone} and \eqref{eq:integration-parts-cutoff} show
that \(\|\chi_A\|_{\widetilde E}<\infty\). The Fatou property
follows by monotone convergence. Finally, if
\[
\sum_{k=1}^\infty\|f_k\|_{\widetilde E}<\infty,
\]
then the inequality
\[
\left(\sum_{k=1}^\infty|f_k|\right)^{**} \leq \sum_{k=1}^\infty
f_k^{**},
\]
together with monotone convergence, gives the usual Riesz--Fischer
argument and hence completeness.

Since
\[
g(1)\|\chi_{(0,1)}\|_{E'}\leq1
\]
whenever \(\|g\|_{E'}\leq1\), \eqref{eq:integration-parts-cutoff},
\eqref{eq:D-cone}, and \eqref{eq:D-equivalence-main} yield
\begin{equation}
\label{eq:normability-interpolation-space} \|f\|_{\widetilde E}
\simeq_{\varphi,E} \|\omega_{\varphi,f}\|_E+\|f\|_1 = \mathcal
N_{\varphi,E}(f).
\end{equation}
Thus \(\mathcal O_\varphi(E)\) is r.i.-normable whenever
\(E\in\Int(L^1,\Lambda_\varphi)\).

We now turn to \(Y_X'\). By \eqref{eq:Q-R-general} and
\eqref{eq:Y-X-general},
\[
\|Q_\varphi|u|\|_{Y_X'} \lesssim \|u\|_X, \qquad u\in X.
\]
Hence, using
\[
f^{**} = \|f\|_1+Q_\varphi\omega_{\varphi,f}
\]
and \(f^*\leq f^{**}\),
\begin{equation*}
 \mathcal O_\varphi(X) \hookrightarrow
Y_X'.
\end{equation*}

Let \(Y\) be any r.i. space containing \(\mathcal O_\varphi(X)\).
For \(g\geq0\),
\[
\omega_{\varphi,Q_\varphi g} = P_\varphi g, \qquad \|Q_\varphi g\|_1
= \int_0^1\varphi(s)g(s)\,ds \leq \|g\|_1.
\]
Thus Lemma~\ref{lem:P-phi} gives \(Q_\varphi:X\to Y\). By duality,
\(R_\varphi:Y'\to X'\), and hence
\[
Y'\hookrightarrow Y_X.
\]
K\"othe duality therefore gives
\[
Y_X'\hookrightarrow Y.
\]
Hence \(Y_X'\) is the least r.i. space containing \(\mathcal
O_\varphi(X)\).

Since \(X_+\in\Int(L^1,\Lambda_\varphi)\),
\eqref{eq:normability-interpolation-space} shows that \(\mathcal
O_\varphi(X_+)\) is r.i.-normable. Moreover, \(X\hookrightarrow
X_+\), and hence
\[
\mathcal O_\varphi(X) \hookrightarrow \mathcal O_\varphi(X_+).
\]
By minimality,
\begin{equation}
\label{eq:Yprime-into-O-Xplus} Y_X' \hookrightarrow \mathcal
O_\varphi(X_+).
\end{equation}

For the reverse inclusion we use the following range estimate:
\begin{equation}
\label{eq:range-stability-general} \|S_\psi(R_\varphi h^*)\|_E
\lesssim_E \|R_\varphi h^*\|_E
\end{equation}
for every r.i. space \(E\). Indeed, set
\[
F(r)=\int_0^rh^*(s)\,ds, \qquad U(r)=\frac{F(r)}{\psi(r)}, \qquad
G(t)=\sup_{r\geq t}U(r).
\]
For \(r\geq t\) and \(r/2\leq s\leq r\), concavity of \(F\) and
monotonicity of \(\psi\) give
\[
U(s)\geq\frac12U(r).
\]
Thus the set on which
\[
U(s)\geq\frac12U(r)
\]
has measure at least \(r/2\geq t/2\), and therefore
\[
U^*(t/2)\geq\frac12U(r).
\]
Taking the supremum over \(r\geq t\), we obtain
\[
G(t)\leq2U^*(t/2).
\]
Since dilations are bounded on every r.i. space, see
\cite{BennettSharpley},
\[
\|G\|_E \lesssim_E \|U\|_E.
\]
Moreover, \(G\) is nonincreasing and \(\psi G\) is nondecreasing: if
\(0<t_1<t_2\), then for \(t_1\leq r<t_2\),
\[
\psi(t_1)U(r) \leq F(r) \leq F(t_2) \leq \psi(t_2)G(t_2),
\]
while for \(r\geq t_2\),
\[
\psi(t_1)U(r) \leq \psi(t_2)G(t_2).
\]
Thus \(S_\psi G=G\). Since \(U\leq G\) and \(G\) is nonincreasing,
\(U^*\leq G\), and therefore
\[
S_\psi U\leq G.
\]
This proves \eqref{eq:range-stability-general}.

Applying it with \(E=X'\), and using
\[
(X_+)'\simeq(X')_-, \qquad \|v\|_{(X')_-} \simeq \|S_\psi v\|_{X'},
\]
we obtain
\begin{equation*}
 \|R_\varphi h^*\|_{(X_+)'} \lesssim
\|R_\varphi h^*\|_{X'}.
\end{equation*}
Therefore, for \(u\in X_+\),
\[
\begin{aligned}
\int_0^1Q_\varphi|u|(t)h^*(t)\,dt &= \int_0^1|u(t)|R_\varphi
h^*(t)\,dt
\\
&\lesssim \|u\|_{X_+}\|R_\varphi h^*\|_{X'},
\end{aligned}
\]
and hence
\begin{equation*}
 Q_\varphi:X_+\longrightarrow Y_X'.
\end{equation*}
The recovery identity now gives
\[
\mathcal O_\varphi(X_+) \hookrightarrow Y_X'.
\]
Together with \eqref{eq:Yprime-into-O-Xplus}, this proves
\[
Y_X' \simeq \mathcal O_\varphi(X_+),
\]
and completes the proof.
\end{proof}

The theorem gives two equivalent descriptions of the same optimal
object. On the associate side,
\[
Y_X = \{h:R_\varphi h^*\in X'\},
\]
whereas on the target side
\[
Y_X' \simeq \mathcal O_\varphi(X_+).
\]
Thus \(X_+\) is not itself the target: it is the canonical
interpolation envelope which generates the optimal Banach target.

\begin{remark}
The preceding theorem goes beyond the classical Hardy--Copson
absorption mechanism; see, for example, \cite{MartinMilman2010}. If
\[
C_\varphi g(t) = \frac1{\varphi(t)} \int_t^1\frac{\varphi(s)}s
g(s)\,ds
\]
is bounded on \(X\), then the recovery identity gives
\[
\mathcal N_{\varphi,X}(f) \simeq \left\| \frac{f^{**}}{\varphi}
\right\|_X.
\]
When \(C_\varphi\) is not bounded, this simpler description need not
hold. Nevertheless, the upper extremal always satisfies
\[
X_+\in\Int(L^1,\Lambda_\varphi),
\]
and therefore
\[
\mathcal O_\varphi(X_+)\simeq Y_X'
\]
is r.i.-normable. Thus failure of the Hardy--Copson absorption does
not prevent the existence of the optimal Banach exterior; the
interpolation envelope \(X_+\) provides it canonically.
\end{remark}
\begin{remark}
There is a natural counterpart from below. Since
\(X_-\in\Int(L^1,\Lambda_\varphi)\), \(\mathcal O_\varphi(X_-)\) is
r.i.-normable, and \(X_-\hookrightarrow X\) gives
\[
\mathcal O_\varphi(X_-) \hookrightarrow \mathcal O_\varphi(X).
\]
Moreover, if \(A\in\Int(L^1,\Lambda_\varphi)\) and
\(A\hookrightarrow X\), then \(A\hookrightarrow X_-\), and the
order-reflection principle yields
\[
\mathcal O_\varphi(A) \hookrightarrow \mathcal O_\varphi(X_-).
\]
Thus \(\mathcal O_\varphi(X_-)\) is maximal among the normable
oscillation classes generated by interpolation spaces contained in
\(X\), while \(Y_X'\) is the least r.i. space containing \(\mathcal
O_\varphi(X)\).
\end{remark}

\section{Normability and interpolation collapse}
\label{sec:collapse}

We now determine when the canonical normable interior and the
optimal Banach exterior coincide:
\[
\mathcal O_\varphi(X_-) \hookrightarrow \mathcal O_\varphi(X)
\hookrightarrow Y_X'.
\]

\begin{theorem}[Normability and interpolation collapse]
\label{thm:main-collapse-general} Let \(\varphi\) be admissible and
let \(X\) be an intermediate r.i. space for \(\vec L_\varphi\). Then
the following are equivalent:
\begin{enumerate}[label=\textup{(\roman*)}]
\item
\(\mathcal N_{\varphi,X}\) is equivalent on \(\mathcal
O_\varphi(X)\) to the norm of an r.i. space;

\item
\[
X\in\Int(L^1,\Lambda_\varphi),
\]
or equivalently,
\[
X'\in\Int(L^\infty,M_\psi);
\]

\item
\[
X_-\simeq X_+,
\]
or equivalently,
\[
(X')_-\simeq (X')_+;
\]

\item
\[
S_\psi:X'\longrightarrow X'
\]
is bounded;

\item
\[
\mathcal O_\varphi(X_-)\simeq Y_X'.
\]
\end{enumerate}
\end{theorem}

\begin{proof}
The equivalence between the two formulations in \textup{(ii)}
follows from \eqref{eq:phi-dual-interpolation}, while
\eqref{eq:phi-extremal-duality} gives
\[
X_-\simeq X_+ \quad\Longleftrightarrow\quad (X')_-\simeq(X')_+.
\]
The equivalence of \textup{(ii)}, \textup{(iii)}, and \textup{(iv)}
then follows from \eqref{eq:pre-collapse}. If \textup{(ii)} holds,
the normability result proved above with \(E=X\) gives \textup{(i)}.

Conversely, assume \textup{(i)} and let \(Y\) denote \(\mathcal
O_\varphi(X)\) endowed with an equivalent r.i. norm. By
Theorem~\ref{thm:optimal-exterior},
\[
\mathcal O_\varphi(X)\hookrightarrow Y_X', \qquad
Y_X'\hookrightarrow Y,
\]
and hence
\[
Y\simeq Y_X' \simeq \mathcal O_\varphi(X_+).
\]
Since both \(X\) and \(X_+\) are intermediate r.i. spaces for \(\vec
L_\varphi\), the order-reflection part of
Theorem~\ref{thm:recovery-general} applies and gives \(X\simeq
X_+\). Hence \textup{(ii)} follows.

Finally,
\[
Y_X'\simeq\mathcal O_\varphi(X_+),
\]
and another application of order reflection gives
\[
\mathcal O_\varphi(X_-)\simeq Y_X' \quad\Longleftrightarrow\quad
X_-\simeq X_+.
\]
Thus \textup{(v)} is equivalent to \textup{(iii)}.
\end{proof}

Thus normability is precisely the collapse of the two
Aronszajn--Gagliardo extremals, whereas outside the collapse regime
\(X_+\) still generates the optimal Banach exterior.

%%%%%%%%%%%%%%%%%%%%%%%%%%%%%%%%%%%%%%%%%%%%%%%%%%%%%%%%%%%%%%%%%%%%%%%%%%%%%%%
\subsection{Fundamental functions and an obstruction to collapse}
%%%%%%%%%%%%%%%%%%%%%%%%%%%%%%%%%%%%%%%%%%%%%%%%%%%%%%%%%%%%%%%%%%%%%%%%%%%%%%%

\begin{corollary}
\label{cor:fundamental-obstruction} Under the assumptions of
Theorem~\ref{thm:main-collapse-general},
\begin{equation}
\label{eq:fundamental-extremal-sandwich} t \lesssim \phi_{X_+}(t)
\lesssim \phi_X(t) \lesssim \phi_{X_-}(t) \lesssim \varphi(t),
\qquad 0<t<1,
\end{equation}
and
\begin{equation}
\label{eq:fundamental-Xplus-general} \phi_{X_+}(t) \simeq
\frac{t}{\|S_\psi\chi_{(0,t)}\|_{X'}}.
\end{equation}
Since
\[
S_\psi\chi_{(0,t)} = \chi_{(0,t)} + \frac{\psi(t)}{\psi(\cdot)}
\chi_{(t,1)},
\]
this gives an explicit formula for \(\phi_{X_+}\).

Consequently, if
\begin{equation}
\label{eq:restricted-S-obstruction} \sup_{0<t<1}
\frac{\|S_\psi\chi_{(0,t)}\|_{X'}} {\phi_{X'}(t)} = \infty,
\end{equation}
then \(X_-\not\simeq X_+\), and hence \(\mathcal N_{\varphi,X}\) is
not r.i.-normable.

If, in addition,
\[
\phi_X(t)\simeq\varphi(t),
\]
then
\begin{equation*}
\phi_{X_-}(t)\simeq\varphi(t),
\end{equation*}
and
\begin{equation}
\label{eq:critical-gap-fundamental}
\frac{\phi_{X_-}(t)}{\phi_{X_+}(t)} \simeq
\frac{\|S_\psi\chi_{(0,t)}\|_{X'}} {\phi_{X'}(t)}.
\end{equation}
\end{corollary}

\begin{proof}
The extremal sandwich gives
\eqref{eq:fundamental-extremal-sandwich}. Moreover,
\[
(X_+)'\simeq(X')_-, \qquad \|\chi_{(0,t)}\|_{(X')_-} \simeq
\|S_\psi\chi_{(0,t)}\|_{X'}.
\]
Since
\[
\phi_E(t)\phi_{E'}(t)=t,
\]
we obtain \eqref{eq:fundamental-Xplus-general} and
\[
\frac{\phi_X(t)}{\phi_{X_+}(t)} \simeq
\frac{\|S_\psi\chi_{(0,t)}\|_{X'}} {\phi_{X'}(t)}.
\]
Thus \eqref{eq:restricted-S-obstruction} rules out collapse. In the
critical case \(\phi_X\simeq\varphi\), the sandwich also gives
\(\phi_{X_-}\simeq\varphi\), and the last identity yields
\eqref{eq:critical-gap-fundamental}.
\end{proof}

%%%%%%%%%%%%%%%%%%%%%%%%%%%%%%%%%%%%%%%%%%%%%%%%%%%%%%%%%%%%%%%%%%%%%%%%%%%%%%%
\subsection{Critical rigidity}
%%%%%%%%%%%%%%%%%%%%%%%%%%%%%%%%%%%%%%%%%%%%%%%%%%%%%%%%%%%%%%%%%%%%%%%%%%%%%%%

At the critical fundamental scale, the collapse criterion becomes
rigid.

\begin{theorem}[Critical rigidity]
\label{thm:critical-rigidity} Let \(\varphi\) be admissible, let
\(X\) be an intermediate r.i. space for \(\vec L_\varphi\), and
assume
\[
\phi_X(t)\simeq\varphi(t), \qquad 0<t<1.
\]
Then
\begin{equation*}
\begin{aligned}
\mathcal N_{\varphi,X}\text{ is r.i.-normable} &\Longleftrightarrow
X\in\Int(L^1,\Lambda_\varphi) \Longleftrightarrow X_-\simeq X_+
\\
&\Longleftrightarrow S_\psi:X'\to X' \Longleftrightarrow
X\simeq\Lambda_\varphi.
\end{aligned}
\end{equation*}
Moreover,
\begin{equation*}
X_-\simeq\Lambda_\varphi.
\end{equation*}
\end{theorem}

\begin{proof}
The first four conditions are equivalent by
Theorem~\ref{thm:main-collapse-general}. Assume \(S_\psi:X'\to X'\)
is bounded. Since
\[
\phi_X(t)\phi_{X'}(t)=t, \qquad \psi(t)=\frac{t}{\varphi(t)},
\]
the critical assumption gives
\begin{equation*}
\phi_{X'}(t)\simeq\psi(t).
\end{equation*}

For \(0<a<1\), let
\[
g_a = \frac1{\psi(a)}\chi_{(0,a)}.
\]
Then \(\|g_a\|_{X'}\simeq1\), while
\[
S_\psi g_a(t) =
\begin{cases}
\psi(a)^{-1},&0<t\leq a,\\
\psi(t)^{-1},&a<t<1.
\end{cases}
\]
The boundedness of \(S_\psi\), followed by Fatou as
\(a\downarrow0\), therefore gives
\begin{equation*}
\frac1\psi\in X'.
\end{equation*}
Hence, for \(h\in M_\psi\),
\[
h^*(t) \leq \frac{\|h\|_{M_\psi}}{\psi(t)},
\]
and so \(M_\psi\hookrightarrow X'\). Conversely, the standard
embedding \(E\hookrightarrow M_{\phi_E}\) for r.i. spaces gives
\[
X' \hookrightarrow M_{\phi_{X'}} \simeq M_\psi;
\]
see \cite{BennettSharpley}. Thus \(X'\simeq M_\psi\), and K\"othe
duality yields \(X\simeq\Lambda_\varphi\).

Finally, Corollary~\ref{cor:fundamental-obstruction} gives
\(\phi_{X_-}\simeq\varphi\), while
\(X_-\in\Int(L^1,\Lambda_\varphi)\). Applying the preceding
implication to \(X_-\) gives
\[
X_-\simeq\Lambda_\varphi.
\]
\end{proof}

%%%%%%%%%%%%%%%%%%%%%%%%%%%%%%%%%%%%%%%%%%%%%%%%%%%%%%%%%%%%%%%%%%%%%%%%%%%%%%%
\section{Applications}
%%%%%%%%%%%%%%%%%%%%%%%%%%%%%%%%%%%%%%%%%%%%%%%%%%%%%%%%%%%%%%%%%%%%%%%%%%%%%%%

\subsection{The Lorentz scale and the Hansson--Brezis--Wainger target}
\label{subsec:Lorentz-applications}

Let
\[
1<p<\infty,\qquad 1\leq q\leq\infty,
\]
and set
\[
X=L^{p,q},\qquad \varphi(t)=t^{1/p},\qquad \psi(t)=t^{1/p'}.
\]
For this power profile we write \(\mathcal O_{1/p}(X)=\mathcal
O_\varphi(X)\) and \(\mathcal N_{1/p,X}=\mathcal N_{\varphi,X}\).
Then
\[
\Lambda_\varphi=L^{p,1}, \qquad M_\psi=L^{p',\infty},
\]
and
\[
S_ph(t) = t^{-1/p'}\sup_{0<s\leq t}s^{1/p'}h^*(s).
\]

Write
\[
\ell(t)=1+\log\frac1t.
\]
For \(1<q\leq\infty\), define
\[
\|f\|_{A_{p,q}} = \left\| \frac1{\ell(t)}
\int_t^1s^{1/p}f^{**}(s)\,\frac{ds}{s} \right\|_{L^q(dt/t)}
\]
and
\[
\|f\|_{\mathcal H_q} = \left\| \frac{f^{**}(t)}{\ell(t)}
\right\|_{L^q(dt/t)}.
\]
We also write
\[
\|f\|_{L^{p,q}(\log L)^{-1}} = \left\| t^{1/p}\ell(t)^{-1}f^*(t)
\right\|_{L^q(dt/t)},
\]
with the usual modifications when \(q=\infty\).

\begin{theorem}
 For \(1<p<\infty\) and \(1\leq
q\leq\infty\),
\begin{equation*}
 (L^{p,q})_-\simeq L^{p,1}.
\end{equation*}
Moreover,
\[
(L^{p,q})_+ \simeq
\begin{cases}
L^{p,1},&q=1,\\
A_{p,q},&1<q\leq\infty.
\end{cases}
\]
For \(1<q\leq\infty\),
\begin{equation*}
\phi_{(L^{p,q})_+}(t) \simeq t^{1/p}\ell(t)^{-1/q'},
\end{equation*}
and
\begin{equation}
\label{eq:Lpq-A-sandwich} L^{p,q} \hookrightarrow A_{p,q}
\hookrightarrow L^{p,q}(\log L)^{-1}.
\end{equation}
Consequently,
\begin{equation*}
(L^{p,q})_-\simeq(L^{p,q})_+ \quad\Longleftrightarrow\quad q=1,
\end{equation*}
so
\[
\left\| t^{-1/p}(f^{**}-f^*) \right\|_{L^{p,q}} +\|f\|_1
\]
is r.i.-normable if and only if \(q=1\).

Finally, the optimal Banach exterior satisfies
\[
Y_{L^{p,1}}'\simeq L^\infty, \qquad Y_{L^{p,q}}'\simeq\mathcal H_q,
\quad 1<q\leq\infty,
\]
and hence, for \(q>1\),
\begin{equation*}
L^{p,q} \longmapsto A_{p,q} \longmapsto \mathcal O_{1/p}(A_{p,q})
\simeq \mathcal H_q.
\end{equation*}
\end{theorem}

\begin{proof}
Since
\[
\phi_{L^{p,q}}(t)\simeq t^{1/p},
\]
Theorem~\ref{thm:critical-rigidity} gives
\[
(L^{p,q})_-\simeq L^{p,1}.
\]
For \(q=1\), \(L^{p,1}\) is an endpoint of \((L^1,L^{p,1})\), and
therefore
\[
(L^{p,1})_+\simeq L^{p,1}.
\]

Let \(1<q\leq\infty\). By the dual description of the upper
extremal,
\[
\|h\|_{((L^{p,q})_+)'} \simeq \|S_ph\|_{L^{p',q'}} \simeq \left\|
\sup_{0<s\leq t}s^{1/p'}h^*(s) \right\|_{L^{q'}(dt/t)}.
\]
To apply \cite[Theorem~4.4]{PustylnikSignes2008}, take
\[
\phi_1(t)=t^{1/p}, \qquad b(t)=\ell(t)^{-1}, \qquad E=L^q(0,\infty).
\]
Here \(\phi_1\) is a quasi-power function, \(b\) is bounded and
increasing with \(b(t^2)\simeq b(t)\), and
\[
B(u)=b(e^{1-1/u})=u, \qquad \pi_B=1.
\]
Moreover, \(E\) has the Fatou property and
\[
\rho_E=\frac1q<1 \quad (q<\infty), \qquad \rho_E=0 \quad (q=\infty).
\]
Thus all the hypotheses of Theorem~4.4 are satisfied. Since
\(\widetilde\phi_1(t)=t/\phi_1(t)=t^{1/p'}\) and \(b(t)\ell(t)=1\),
its dual formula gives precisely
\[
\|h\|_{(A_{p,q})'} \simeq \left\| \sup_{0<s\leq t}s^{1/p'}h^*(s)
\right\|_{L^{q'}(dt/t)}.
\]
Hence
\[
(L^{p,q})_+\simeq A_{p,q}.
\]
Moreover, \cite[Proposition~3.1]{PustylnikSignes2008} gives
\eqref{eq:Lpq-A-sandwich}, while a direct computation on
\(\chi_{(0,a)}\) yields
\[
\phi_{A_{p,q}}(a) \simeq a^{1/p}\ell(a)^{-1/q'}.
\]
Thus \(L^{p,q}\not\simeq(L^{p,q})_+\) for \(q>1\), whereas the
collapse holds for \(q=1\). The normability statement follows from
Theorem~\ref{thm:main-collapse-general}.

By the definition of \(Y_X\),
\[
\|h\|_{Y_{L^{p,q}}} = \|t^{1/p}h^{**}(t)\|_{L^{p',q'}}.
\]
Since \(t\mapsto t h^{**}(t)\) is quasiconcave, the standard Lorentz
representation for quasiconcave functions gives
\begin{equation*}
 \|t^{1/p}h^{**}(t)\|_{L^{p',q'}} \simeq
\|t\,h^{**}(t)\|_{L^{q'}(dt/t)}.
\end{equation*}
For \(q=1\), the right-hand side is
\[
\sup_{0<t<1}t h^{**}(t)=\|h\|_1,
\]
so
\[
Y_{L^{p,1}}\simeq L^1, \qquad Y_{L^{p,1}}'\simeq L^\infty.
\]

For \(1<q\leq\infty\), \(\mathcal H_q\) is the space
\(L(\infty,q;-1,0)\) in the notation of \cite{OpicPick}. By
\cite[Theorems~3.8 and~6.11]{OpicPick},
\[
\|h\|_{(\mathcal H_q)'} \simeq \|t\,h^{**}(t)\|_{L^{q'}(dt/t)},
\]
with the usual interpretation when \(q=\infty\). Hence
\[
Y_{L^{p,q}} \simeq (\mathcal H_q)',
\]
and K\"othe duality gives
\[
Y_{L^{p,q}}' \simeq \mathcal H_q.
\]
\end{proof}

\begin{remark}
The supremal criterion detects limiting interpolation spaces which
are not covered by the classical Copson mechanism. Fix \(\beta>0\).
Choose \(0<t_\beta<1\) so that
\[
t^{1/p}\ell(t)^{-\beta}
\]
is increasing and concave on \((0,t_\beta]\), and let
\(\theta_\beta\) be an increasing concave extension to \((0,1)\)
which agrees with this function on \((0,t_\beta]\). Thus
\[
\theta_\beta(t) \simeq t^{1/p}\ell(t)^{-\beta}, \qquad 0<t<1.
\]
Set
\[
X_\beta=\Lambda_{\theta_\beta}, \qquad
\eta_\beta(t)=\frac{t}{\theta_\beta(t)}.
\]
Then
\[
L^{p,1}\hookrightarrow X_\beta\hookrightarrow L^1, \qquad
X_\beta'\simeq M_{\eta_\beta},
\]
and
\[
\eta_\beta(t) \simeq t^{1/p'}\ell(t)^\beta.
\]

Since \(A_{\theta_\beta}<\infty\), the \(h^*\)-description of the
Marcinkiewicz norm applies. Hence
\[
\begin{aligned}
\eta_\beta(t)S_ph(t) &\lesssim \ell(t)^\beta \sup_{0<s\leq
t}s^{1/p'}h^*(s)
\\
&\leq \sup_{0<s\leq t} \ell(s)^\beta s^{1/p'}h^*(s) \lesssim
\|h\|_{M_{\eta_\beta}},
\end{aligned}
\]
and therefore
\[
S_p:X_\beta'\longrightarrow X_\beta'
\]
is bounded. Theorem~\ref{thm:main-collapse-general} gives
\[
X_\beta\in\Int(L^1,L^{p,1}),
\]
and consequently
\[
\left\| t^{-1/p}(f^{**}-f^*) \right\|_{X_\beta} +\|f\|_1
\]
is r.i.-normable.

On the other hand, the corresponding Copson operator
\[
C_pg(t) = t^{-1/p}\int_t^1s^{1/p-1}g(s)\,ds
\]
is not bounded on \(X_\beta\). Indeed, let \(0<a<t_\beta/2\) and
\(g_a=\chi_{(0,a)}\). Then
\[
\|g_a\|_{X_\beta} = \theta_\beta(a) \simeq a^{1/p}\ell(a)^{-\beta},
\]
while, for \(0<t<a/2\),
\[
C_pg_a(t) = p\,t^{-1/p} \bigl(a^{1/p}-t^{1/p}\bigr) \gtrsim
a^{1/p}t^{-1/p}.
\]
Since
\[
\theta_\beta(t) = t^{1/p}\ell(t)^{-\beta}
\]
for \(0<t<t_\beta\),
\[
d\theta_\beta(t) \simeq t^{1/p-1}\ell(t)^{-\beta}\,dt
\]
there, and hence
\[
\|C_pg_a\|_{X_\beta} \gtrsim a^{1/p} \int_0^{a/2}
\frac{dt}{t\,\ell(t)^\beta}.
\]
For \(0<\beta\leq1\) the last integral is infinite, whereas for
\(\beta>1\),
\[
\int_0^{a/2} \frac{dt}{t\,\ell(t)^\beta} \simeq \ell(a)^{1-\beta},
\]
so
\[
\frac{\|C_pg_a\|_{X_\beta}}
     {\|g_a\|_{X_\beta}}
\gtrsim \ell(a) \longrightarrow\infty \qquad (a\downarrow0).
\]
Thus
\[
C_p:X_\beta\not\longrightarrow X_\beta
\]
for every \(\beta>0\). Hence normability may persist even when the
classical \(f^{**}\)-absorption mechanism fails.
\end{remark}

%%%%%%%%%%%%%%%%%%%%%%%%%%%%%%%%%%%%%%%%%%%%%%%%%%%%%%%%%%%%%%%%%%%%%%%%%%%%%%%
\subsection{The Orlicz scale at the critical power}
%%%%%%%%%%%%%%%%%%%%%%%%%%%%%%%%%%%%%%%%%%%%%%%%%%%%%%%%%%%%%%%%%%%%%%%%%%%%%%%

The Orlicz scale behaves differently: logarithmic perturbations at
the critical power do not restore normability.

Let \(A\) be a Young function, let \(\widetilde A\) be its
complementary function, and recall that
\[
(L^A)'\simeq L^{\widetilde A}, \qquad \phi_{L^A}(t) \simeq
\frac1{A^{-1}(1/t)}.
\]

\begin{proposition}
Let \(1<p<\infty\), \(\varphi(t)=t^{1/p}\), and assume
\[
L^{p,1}\hookrightarrow L^A\hookrightarrow L^1.
\]
If
\[
\phi_{L^A}(t)\simeq t^{1/p},
\]
then \(\mathcal N_{1/p,L^A}\) is not r.i.-normable.

More generally, let
\[
A_\gamma(u) \simeq u^p\bigl(\log(e+u)\bigr)^\gamma, \qquad
u\to\infty, \qquad \gamma\leq0.
\]
Then
\[
L^{p,1}\hookrightarrow L^{A_\gamma}\hookrightarrow L^1,
\]
and \(\mathcal N_{1/p,L^{A_\gamma}}\) is not r.i.-normable.
\end{proposition}

\begin{proof}
If
\[
\phi_{L^A}(t)\simeq t^{1/p},
\]
then
\[
A^{-1}(s)\simeq s^{1/p}, \qquad s\to\infty,
\]
and therefore
\[
A(u)\simeq u^p, \qquad u\to\infty.
\]
Hence \(L^A\simeq L^p\) on \((0,1)\). If \(\mathcal N_{1/p,L^A}\)
were r.i.-normable, Theorem~\ref{thm:critical-rigidity} would give
\[
L^A\simeq L^{p,1},
\]
which is impossible since
\[
L^{p,1}\subsetneq L^p.
\]

Consider now the logarithmically perturbed scale. Its complementary
Young function satisfies
\[
\widetilde A_\gamma(v) \simeq v^{p'}
\bigl(\log(e+v)\bigr)^{-\gamma/(p-1)}, \qquad v\to\infty.
\]
For any fixed \(r>p'\),
\[
M(p')\cap M(r) \hookrightarrow L^{\widetilde A_\gamma}
\hookrightarrow M(p'),
\]
where \(M(q)=L^{q,\infty}(0,1)\). Moreover,
\[
t^{-1/p'}\notin L^{\widetilde A_\gamma}(0,1), \qquad \int_1^\infty
\frac{\widetilde A_\gamma(s)}{s^{r+1}}\,ds <\infty.
\]

Since \(\psi(t)=t^{1/p'}\), we write \(S_p=S_\psi\), so that
\[
S_ph(t) = t^{-1/p'} \sup_{0<s\leq t}s^{1/p'}h^*(s).
\]
After an immaterial modification of \(\widetilde A_\gamma\) near the
origin, all the hypotheses of \cite[Theorem~A]{KermanPhippsPick2014}
are satisfied with \(b=1\), \(\alpha=p'\), \(\beta=r\), and
\(A_1=A_2=\widetilde A_\gamma\). With these choices, the operator
\(S_\alpha\) in that theorem is precisely \(S_p\). Consequently,
boundedness of
\[
S_p: L^{\widetilde A_\gamma} \longrightarrow L^{\widetilde A_\gamma}
\]
would require the existence of \(K>0\) such that
\begin{equation}
\label{eq:Orlicz-supremal-condition} \int_1^v \frac{\widetilde
A_\gamma(s)}{s^{p'+1}}\,ds \lesssim \frac{\widetilde
A_\gamma(Kv)}{v^{p'}}, \qquad v>1.
\end{equation}

Put
\[
\delta=\frac{\gamma}{p-1}.
\]
For large \(v\),
\[
\int_1^v \frac{\widetilde A_\gamma(s)}{s^{p'+1}}\,ds \simeq \int_1^v
\frac{ds}{s(\log(e+s))^\delta}.
\]
Restricting the last integral to \((v^{1/2},v)\) gives
\[
\int_1^v \frac{\widetilde A_\gamma(s)}{s^{p'+1}}\,ds \gtrsim (\log
v)^{1-\delta},
\]
whereas, for each fixed \(K>0\),
\[
\frac{\widetilde A_\gamma(Kv)}{v^{p'}} \simeq (\log v)^{-\delta}.
\]
Thus the quotient of the left-hand side of
\eqref{eq:Orlicz-supremal-condition} by its right-hand side is
bounded from below by a constant multiple of \(\log v\), and
therefore tends to infinity. Hence
\[
S_p: L^{\widetilde A_\gamma} \not\longrightarrow L^{\widetilde
A_\gamma}
\]
for every \(\gamma\leq0\).

Since
\[
(L^{A_\gamma})' \simeq L^{\widetilde A_\gamma},
\]
Theorem~\ref{thm:main-collapse-general} gives the conclusion.
\end{proof}

\begin{remark}[Lorentz versus Orlicz at the critical power]
The preceding examples show that the critical power index alone does
not determine normability. Logarithmic perturbations in the Lorentz
endpoint scale may satisfy the supremal criterion even though Copson
absorption fails, whereas the corresponding Orlicz perturbations
remain on the non-normable side of the criterion. Thus \(S_\psi\)
detects fine interpolation structure which is invisible at the level
of power indices. A complementary computational approach to Orlicz
interpolation has recently been developed in
\cite{GogatishviliKermanSpektor2026}, where gauge functionals built
from Hardy averages are used to obtain more explicit descriptions of
dual norms arising from the \(K\)-method.
\end{remark}

%%%%%%%%%%%%%%%%%%%%%%%%%%%%%%%%%%%%%%%%%%%%%%%%%%%%%%%%%%%%%%%%%%%%%%%%%%%%%%%

\end{document}